\documentclass[a4paper,12pt]{article}

\usepackage{amsmath,amsfonts,amssymb}
\usepackage{amsthm}
\usepackage{color,multicol}
\usepackage{graphicx}

\newtheorem{lemma}{Lemma}[section]
\newtheorem{theorem}[lemma]{Theorem}
\newtheorem{corollary}[lemma]{Corollary}

\newtheorem{definition}[lemma]{Definition}

\newtheorem{remark}[lemma]{Remark}

\begin{document}

\title{On the discrete spectrum of Schr\"odinger operators with Ahlfors regular potentials in a strip}

\author{Martin Karuhanga\footnote{Department of Mathematics, Mbarara University of Science and Technology, P.O BOX 1410, Mbarara, Uganda, E-mail: mkaruhanga@must.ac.ug, ORCID : 0000-0002-7254-9073}}

\date{}

\maketitle

\begin{abstract}
In this paper, quantitative upper estimates for the number of eigenvalues lying below the essential spectrum of
Schr\"odinger operators with potentials generated by Ahlfors regular measures in a strip subject to two different types of boundary conditions (Robin and Dirichlet respectively) are presented. The estimates are presented in terms of weighted $L^1$ norms and Orlicz norms of the potential.
\end{abstract}\noindent

\section{Introduction}
Let $V\in L^1_{\textrm{loc}}(\mathbb{R}^d)$. According to the Cwikel-Lieb-Rozenblum (CLR) inequality (see, e.g., \cite{BEL,Roz}), the number of negative eigenvalues of the Schr\"odinger operator $-\Delta - V, V\ge 0$ on $L^2(\mathbb{R}^d)$ with $d\ge 3$ is estimated above by $\|V\|_{L^{\frac{d}{2}}(\mathbb{R}^d)}$. In the case $d = 2$, this estimate fails, but there has been significant recent progress in obtaining estimates of the CLR-type in two-dimensions and the best known estimates have been obtained in \cite{Eugene}. Estimates for the number of negative eigenvalues of Schr\"odinger operators with potential of the form $V\mu$, where $\mu$ is a Radon measure and $V$ is an appropriate function, were obtained in \cite{Bra}, and results from \cite{Eugene} were extended to this setting in \cite{ME, KS}.  In the present paper, we obtain estimates for the number of eigenvalues below the essential spectrum of a Schr\"odinger operator with potential of the form $V\mu$ similar to those in \cite{KS} in a strip subject to boundary conditions of the Robin type. Similar estimates are also obtained when the domain of the operator is characterized by Dirichlet boundary conditions (see remark \ref{dir}). Below is a precise description of the operator studied herein.\\

Let $S := \mathbb{R}\times (0, a),\, a > 0$ be a strip, $\mu$ a $\sigma$-finite positive Radon measure on $\mathbb{R}^2$ and $V : \mathbb{R}^2 \longrightarrow \mathbb{R}$  a non-negative function integrable on bounded subsets of $S$ with respect to $\mu$.
We study the following Schr\"odinger operator
\begin{equation}\label{H}
H_{\mu} := -\Delta - V\mu\,,\;\;V\ge 0, \;\;\;\;\;\textrm{on}\;\;L^2(S),
\end{equation}
where $\Delta := \sum^2_{k = 1}\frac{\partial^2}{\partial x^2_k}$, subject to the following Robin boundary conditions\begin{equation}\label{R}
u_{x_2}(x_1, 0) + \alpha u(x_1, 0) = u_{x_2}(x_1,a) + \beta u(x_1, a) = 0,
\end{equation} where $\alpha, \beta \in \mathbb{R}$. Here $u_{x_i} (i = 1,2)$ denotes the partial derivative of $u$ with respect to $x_i$ and
note that $\mu$ does not have to be the two-dimensional Lebesgue measure. A physical motivation to this problem is closely related to the study of spectral properties of quantum waveguides (see, e.g., \cite{TE, PA, KK,Rad, AFR}).\\\\ Under certain assumptions about $V$ and $\mu$, \eqref{H} is well defined and self-adjoint on $L^2(S)$ and its essential spectrum is the interval $[\lambda_1, +\infty)$, where $\lambda_1$ is the first eigenvalue of $-\Delta = -\partial^2_{x_2}$ considered along the width of the strip with boundary conditions \eqref{R}
(see, e.g., \cite{Aize,Her,KK,Herc}). For a detailed discussion about
what $\lambda_1$ is, see e.g. \cite{Kar}.
The case where both $\alpha$ and $\beta$ are equal to zero with $\mu$ being the two-dimensional Lebesgue measure has been previously studied by A. Grigor'yan and N. Nadirashvili \cite{Grig} who obtained estimates in terms of weighted $L^1$ norms and $L^p, p> 1$ norms of $V$.
Below, we prove stronger results in our more general setting. \\\\
Let $\mathcal{H}$ be a Hilbert space and let $\mathbf{q}$ be a Hermitian form with a domain
$\mbox{Dom}\, (\mathbf{q}) \subseteq \mathcal{H}$. Set
\begin{equation}\label{hermitian}
N_- (\mathbf{q}) := \sup\left\{\dim \mathcal{L}\, | \  \mathbf{q}[u] < 0, \,
\forall u \in \mathcal{L}\setminus\{0\}\right\} ,
\end{equation}
where $\mathcal{L}$ denotes a linear subspace of $\mbox{Dom}\, (\mathbf{q})$.  The number $N_- (\mathbf{q})$ is called the Morse index of $\mathbf{q}$ in $\mbox{Dom}\,(\mathbf{q})$. If $\mathbf{q}$
is the quadratic form of a self-adjoint operator $A$ with no essential spectrum in $(-\infty, 0)$, then
$N_- (\mathbf{q})$ is the number of negative eigenvalues of $A$ repeated according to their
multiplicity (see, e.g., \cite[S1.3]{BerShu} or \cite[Theorem 10.2.3]{BirSol}).\\
Estimating the number of eigenvalues of \eqref{H} below its essential spectrum is equivalent to estimating the number of negative eigenvalues of the operator
\begin{equation}\label{H*}
H_{\lambda_1, \mu} = -\Delta - \lambda_1 -V\mu \mbox{ on } L^2(S)
\end{equation}
subject to boundary conditions in \eqref{R}. Now, defining \eqref{H*} via its quadratic form we have
 \begin{eqnarray}\label{S}
\mathcal{E}_{\lambda_1, \mu, S}[u] &:=& \int_S |\nabla u(x)|^2\,dx - \lambda_1\int_{S}|u(x)|^2\,dx - \alpha\int_{\mathbb{R}}|u(x_1, 0)|^2\,dx_1\nonumber \\&+& \beta\int_{\mathbb{R}}|u(x_1, a)|^2\,dx_1 - \int_{\overline{S}} V(x)|u(x)|^2\,d\mu(x),\\
\textrm{Dom}\left( \mathcal{E}_{\lambda_1, \mu, S}\right) &=& \left\{u\in W^1_2(S)\cap L^2(\overline{S}, Vd\mu\right\}\nonumber.
\end{eqnarray}
Note that we take the closure $\overline{S}$ of the open strip $S$ in the terms involving $\mu$ as this measure might charge
subsets of the horizontal lines $x_2 = 0$ and $x_2 = a$.\\
We denote by $N_-\left(\mathcal{E}_{\lambda_1, \mu, S} \right)$ the number of negative eigenvalues of \eqref{H*} counting multiplicities.\\\\

Let $ S_n := (n, n + 1)\times (0, a),\,\,n \in\mathbb{Z}.$ Then it follows from the variational principle \cite[Lemma 1.6.2]{ME} (see also \cite[Ch.6, $\S$2.1, Theorem 4]{Cou} for the case when $\mu$ is absolutely continuous with respect to the Lebesgue measure) that
\begin{equation}\label{variation}
N_-\left(\mathcal{E}_{\lambda_1, \mu, S} \right) \le \sum_{n\in\mathbb{Z}}N_-\left(\mathcal{E}_{\lambda_1, \mu, S_n} \right),
\end{equation}
where $N_-\left(\mathcal{E}_{\lambda_1, \mu, S_n} \right)$ are the restrictions of the form $\mathcal{E}_{\lambda_1, \mu, S} $ to $\overline{S_n}$.
Let $u(x) = u_1(x_2)$, where $u_1$ is an eigenfunction of $-\Delta = -\partial^2_{x_2}$ considered along the width of the strip with boundary conditions \eqref{R} corresponding to the first eigenvalue $\lambda_1$. Then it is easy to see that
$$
\mathcal{E}_{\lambda_1, \mu, S_n}[u] = - \int_{\overline{S_n}} V(x)|u_1(x_2)|^2\,d\mu(x)
$$
and since $V \ge 0$, the right-hand side is strictly negative unless $$\mu\Big(\mbox{supp} \big(V|u_1|^2\big)\cap\overline{S_n}\Big) = 0.$$
So, one usually has $N_-\left(\mathcal{E}_{\lambda_1, \mu,S_n}\right) \ge 1$ and thus the right-hand side of \eqref{variation} diverges. To avoid this, we shall split the problem  into two problems. The first will be defined by the restriction of the form to the subspace of functions
obtained by multiplying $u_1(x_2)$ by functions depending only on $x_1$, and is thus reduced to a well studied one-dimensional Schr\"odinger operator. The second problem will be defined by a class of functions orthogonal to $u_1$ in the $L^2((0, a))$ inner product.

\section{Notation}\label{notation}
In order to state the estimate for $N_-\left(\mathcal{E}_{\lambda_1, \mu, S} \right)$, we need some notation from the theory of Orlicz spaces (see, e.g., \cite{Ad, KR, RR}). Let $\Phi$ and $\Psi$ be mutually complementary $N$-functions, and let $L_\Phi(\Omega,\mu)$,
$L_\Psi(\Omega, \mu)$ be the corresponding Orlicz spaces.  We will use the following
norms on $L_\Psi(\Omega, \mu)$
\begin{equation}\label{Orlicz}
\|f\|_{\Psi, \mu} = \|f\|_{\Psi, \Omega, \mu} = \sup\left\{\left|\int_\Omega f g d\mu\right| : \
\int_\Omega \Phi(|g|) d\mu \le 1\right\}
\end{equation}
and
\begin{equation}\label{Luxemburg}
\|f\|_{(\Psi, \mu)} = \|f\|_{(\Psi, \Omega, \mu)} = \inf\left\{\kappa > 0 : \
\int_\Omega \Psi\left(\frac{|f|}{\kappa}\right) d\mu \le 1\right\} .
\end{equation}
These two norms are equivalent
\begin{equation}\label{Luxemburgequiv}
\|f\|_{(\Psi, \mu)} \le \|f\|_{\Psi, \mu} \le 2 \|f\|_{(\Psi, \mu)}\, , \ \ \ \forall f \in L_\Psi(\Omega),
\end{equation}(see \cite{Ad}).\\
Note that
\begin{equation}\label{LuxNormImpl}
\int_\Omega \Psi\left(\frac{|f|}{\kappa_0}\right) d\mu \le C_0, \ \ C_0 \ge 1  \ \ \Longrightarrow \ \
\|f\|_{(\Psi)} \le C_0 \kappa_0 .
\end{equation}
Indeed, since $\Psi$ is  convex and increasing on
$[0, +\infty)$, and $\Psi(0) = 0$, we get for any $\kappa \ge C_0 \kappa_0$,
\begin{equation}\label{LuxProof}
\int_{\Omega} \Psi\left(\frac{|f|}{\kappa}\right) d\mu \le
\int_{\Omega} \Psi\left(\frac{|f|}{C_0 \kappa_0}\right) d\mu \le
\frac{1}{C_0} \int_{\Omega} \Psi\left(\frac{|f|}{\kappa_0}\right) d\mu \le 1
\end{equation}
(see \cite{Eugene}).
It follows from \eqref{LuxNormImpl} with $\kappa_0 = 1$ that
\begin{equation}\label{LuxNormPre}
\|f\|_{(\Psi, \mu)} \le \max\left\{1, \int_{\Omega} \Psi(|f|) d\mu\right\} .
\end{equation}
We will also need the following equivalent norm on $L_{\Psi}(\Omega)$ with $\mu(\Omega) < \infty$, which was introduced in \cite{Sol}:
\begin{equation}\label{new}
\|f\|^{(av)}_{\Psi, \Omega} := \sup\left\{\left|\int_\Omega f g\, d\mu\right| : \
\int_\Omega \Phi(|g|) d\mu \le \mu(\Omega)\right\}.
\end{equation}

We will use the following pair of mutually complementary $N$-functions
$$
\mathcal{A}(s) = e^{|s|} - 1 - |s| , \ \ \ \mathcal{B}(s) = (1 + |s|) \ln(1 + |s|) - |s| , \ \ \ s \in \mathbb{R} .
$$
\begin{definition}
 Let $\mu$ be a positive Radon measure on $\mathbb{R}^2$. We say the measure $\mu$ is Ahlfors regular of dimension $d > 0$ if there exist positive constants $c_0$ and $c_1$ such that
\begin{equation}\label{Ahlfors}
c_0r^{d} \le \mu(B(x, r)) \le c_1r^{d}\;
\end{equation}for all $0< r \le$ diam(supp$\,\mu$) and all $x\in \textrm{supp}\,\mu$, where $B(x, r)$ is the ball of radius $r$ centred at $x$, and the constants $c_0$ and $c_1$ are independent of the balls.
\end{definition}

\begin{definition}{\rm (Local Ahlfors regularity)}
We say that a measure $\mu$ is locally Ahlfors regular on a bounded set $G \subset \mathbb{R}^2$ if for every $R < \infty$
there exist $d > 0$ and  positive constants
$c_0(R)$ and $c_1(R)$ such that
\begin{equation}\label{AhlforsR}
c_0(R)\, r^{d} \le \mu(B(x, r)) \le c_1(R)\, r^{d}
\end{equation}
for all $0< r \le R$ and all $x\in \textrm{supp}\,\mu \cap \overline{G}$. We say that $\mu$ is locally Ahlfors regular on the strip
$S$ if \eqref{AhlforsR} holds for all $0< r \le R$ and all $x\in \textrm{supp}\,\mu \cap \overline{S}$,
and there exist constants $c_2, c_3 > 0$ such that
\begin{equation}\label{locAhlfors}
c_2\mu\left(\overline{S_{n \pm 1}}\right) \le \mu\left(\overline{S_n}\right)\le c_3\mu\left(\overline{S_{n \pm 1}}\right),\,\,\,\,\forall  n\in\mathbb{Z}\,.
\end{equation}
\end{definition}

Thus for each $n\in\mathbb{Z}$,
\begin{equation}\label{locAhlfors*}
c^k_2\mu\left(\overline{S_{n \pm k}}\right) \le \mu\left(\overline{S_n}\right)\le c^k_3\mu\left(\overline{S_{n \pm k}}\right),\,\,\,\,\forall  k\in\mathbb{N}\,.
\end{equation}

From now onwards, it will be assumed that $\mu$ is a $\sigma$-finite positive Radon measure that is locally Ahlfors regular on $S$.
\section{Statement of the main result}\label{S3}
 Let
\begin{eqnarray*}
I_n := [2^{n - 1}, 2^n], \ n > 0 , \ \ \ I_0 := [-1, 1] , \ \ \
I_n := [-2^{|n|}, -2^{|n| - 1}], \ n < 0 ,
\end{eqnarray*}
\begin{eqnarray*}
\mathcal{F}_n &:=& \int_{I_n}\int_0^a|x_1|V(x)|u_1(x_2)|^2\,d\mu(x)\,\,\,\, n\neq 0\,,\\ \mathcal{F}_0 &:=& \int_{I_0}\int_0^aV(x)|u_1(x_2)|^2\,d\mu(x)\,,\\
M_n &:=& \|V\|_{\mathcal{B}, \overline{S_n},\mu}\,,
\end{eqnarray*}
where $u_1$ is a normalized eigenfunction of $-\Delta = -\partial^2_{x_2}$ on $(0, a)$ with boundary conditions \eqref{R} corresponding to the first eigenvalue $\lambda_1$ (here the normalization is with respect to the Lebesgue measure).
\begin{theorem}\label{rbthm}
Let $\mu$ be a $\sigma$-finite positive Radon measure on $\mathbb{R}^2$ that is locally Ahlfors regular on $S$ and $V\in L_{\mathcal{B}}(S_n, \mu)$ for every $n\in\mathbb{Z}$. Then there exist constants $C, c >0$ such that
\begin{equation}\label{rbtheqn}
N_-\left(\mathcal{E}_{\lambda_1, \mu, S} \right) \le 1 + C\left(\underset{\{\mathcal{F}_n > c,\,n\in\mathbb{Z}\}}\sum \sqrt{\mathcal{F}_n} + \underset{\{M_n > c,\,n\in\mathbb{Z}\}}\sum M_n \right).
\end{equation}
\end{theorem}

\section{Auxiliary results}
We start with a result that was obtained in \cite{KS} (see also \cite[Lemma 3.1.1]{ME}). For the reader's convenience, we reproduce the
proof here.
\begin{lemma}\label{direction}
Let $\mu$ be a $\sigma$-finite Radon measure on $\mathbb{R}^2$ such that $\mu(\{x\})= 0$ for all $x\in\mathbb{R}^2$. Let
\begin{equation}\label{sim}
\Sigma := \left\{\theta \in [0, \pi)\;:\;\exists\;l_{\theta}\mbox{ such that }\mu(l_{\theta}) > 0\right\},
\end{equation}
where $l_{\theta}$ is a line in $\mathbb{R}^2$ in the direction of the vector $(\cos\theta, \sin\theta)$. Then $\Sigma$ is at most countable.
\end{lemma}
\begin{proof}
Let $$
\Sigma_N := \left\{\theta \in [0, \pi)\;:\;\exists\; \l_{\theta} \mbox{ such that } \mu(l_{\theta}\cap B(0, N)) > 0\right\},
$$
where $B(0, N)$ is the ball of radius $N\in\mathbb{N}$ centred at $0$. Then
$$
\Sigma = \underset{N\in\mathbb{N}}\cup \Sigma_N.
$$It is now enough to show that $\Sigma_N$ is at most countable for $\forall N\in\mathbb{N}$. Suppose that $\Sigma_N$ is uncountable. Then there exists a $\delta > 0$ such that
$$
\Sigma_{N,\delta} := \left\{\theta \in [0, \pi)\;:\;\exists\; \l_{\theta} \mbox{ such that }\mu(l_{\theta}\cap B(0, N)) > \delta\right\}
$$ is infinite.
Otherwise, $\Sigma_N = \underset{n\in\mathbb{N}}\cup \Sigma_{N, \frac{1}{n}}$ would have been finite or countable. Now take $\theta_1,..., \theta_k,... \in\Sigma_{N, \delta}$. Then
$$
\mu\left(l_{\theta_k}\cap B(0, N)\right) > \delta,\;\;\;\forall k\in\mathbb{N}\,.
$$
Since $l_{\theta_j}\cap l_{\theta_k} ,\;j\neq k$ contains at most one point, then $$\mu\left(\underset{j \neq k}\cup (l_{\theta_j}\cap l_{\theta_k})\right) = 0.$$
Let
$$
\tilde{l}_{\theta_k}:= l_{\theta_k}\backslash\underset{j \neq k}\cup (l_{\theta_j}\cap l_{\theta_k})\,.
$$Then $\tilde{l}_{\theta_j}\cap\tilde{l}_{\theta_k} = \emptyset,\;j \neq k$ and $\tilde{l}_{\theta_k}\cap B(0, N) \subset B(0, N)$.
So
$$
\mu\left(\underset{k\in\mathbb{N}}\cup (\tilde{l}_{\theta_k}\cap B(0, N))\right) \le \mu\left(B(0, N)\right) < \infty\,.
$$ But
$$
\mu\left(\tilde{l}_{\theta_k}\cap B(0, N)\right) = \mu\left(l_{\theta_k}\cap B(0, N)\right) \ge \delta
$$
which implies
$$
\underset{k\in\mathbb{N}}\sum \mu\left( \tilde{l}_{\theta_k}\cap B(0, N)\right) \geq \underset{k\in\mathbb{N}}\sum\delta = \infty\,.
$$
This contradiction means that $\Sigma_N$ is at most countable for each $N\in\mathbb{N}$. Hence $\Sigma$ is at most countable.
\end{proof}
\begin{corollary}\label{cor-direct}
There exists $\theta_0 \in [0, \pi)$ such that $\theta_0 \notin \Sigma$ and $\theta_0 + \frac{\pi}{2} \notin \Sigma$.
\end{corollary}
\begin{proof}
The set
$$
\Sigma - \frac{\pi}{2} := \left\{ \theta - \frac{\pi}{2} \;: \theta\in \Sigma\right\}
$$ is at most countable. This implies that there exists a $\theta_0\notin \Sigma \cup (\Sigma - \frac{\pi}{2})$. Thus $\theta_0 + \frac{\pi}{2}\notin \Sigma$.
\end{proof}

Let $G\subset\mathbb{R}^2$ be a bounded set with Lipschitz boundary such that $0 < \mu(\overline{G}) < \infty$. Let $G_*$ be the smallest square containing $G$ with sides chosen in the directions $\theta_0$ and $\theta_0 + \frac{\pi}{2}$ from Corollary \ref{cor-direct}, and let
$G^*$ be the closed square with the same centre as $G_*$ and  sides in the same direction but of length $3$ times that of $G_*$. Let
$$
\kappa_0 (G) := \frac{\mu(G^*)}{\mu\left(\overline{G}\right)}\,.
$$

There exists a bounded linear operator $$T: W^1_2(G) \longrightarrow W^1_2(\mathbb{R}^2)$$  which satisfies
$$
Tu|_G = u, \;\;\;\forall u \in W^1_2(G)
$$
(see, e.g.,  \cite[Ch.VI, Theorem 5]{Stein}).
We will use the following notation:
$$
u_E := \frac{1}{|E|}\int_E u(x)\,dx,
$$ where $E \subseteq\mathbb{R}^2$ is a set of a finite two dimensional Lebesgue measure $|E|$.

The following result is similar to \cite[Lemma 3.2.13]{ME} and follows directly from the proof of the latter.
\begin{lemma}\label{measlemma4}
Let $G$ be as above and $\mu$ be a $\sigma$-finite positive  Radon measure on $\mathbb{R}^2$ that is locally Ahlfors regular on $G$. Choose and fix a direction satisfying Corollary \ref{cor-direct}. Further, for all $x\in\overline{G}$ and for all $r > 0$, let $\Delta_x(r)$ be a square with edges of length $r$ in the chosen direction centred at $x\in \textrm{supp}\,\mu \cap \overline{G}$. Then for any
$V\in L_{\mathcal{B}}\left(\overline{G}, \mu\right), \; V\geq 0$ and any $m\in \mathbb{N}$ there exists a finite cover of $\textrm{supp}\,\mu\cap \overline{G}$  by squares $Q_{x_k}(r_{x_k}), r_{x_k} > 0, k = 1, 2, ..., m_0$,  such that $m_0\le m$  and
\begin{equation}\label{maz9}
\int_{\overline{G}} V(x)|u(x)|^2d\mu(x) \le C(G) C(d)\frac{c_1(R)}{c_0(R)}\kappa_0^2\,
m^{-1}\|V\|^{(av)}_{\mathcal{B}, \overline{G}, \mu}\|u\|^2_{W^1_2(G)}
\end{equation}for all $u\in W^1_2(G)\cap C(\overline{G})$ with $(Tu)_{Q_{x_k}(r_{x_k})} = 0, k = 1,..., m_0$, where the constant $C(G)$
depends only on $G$ and is invariant under parallel translations of $G$, $C(d)$ depends only on $d$ in \eqref{AhlforsR},
and $R$ is the diameter of $G^*$. If $m = 1$, one can take $m_0 = 0$.
\end{lemma}

Let
\begin{equation}\label{n}
\mathcal{E}_{S_n}[u] := \int_{S_n} \mid\nabla u(x)\mid^2 dx + \beta\int_n^{n + 1}\mid u(x_1, a)\mid^2 dx_1 - \alpha \int_n^{n + 1}\mid u(x_1, 0)\mid^2 dx_1,
\end{equation} for all $u\in W^1_2(S_n)$.\\
Let $\lambda_1 \le \lambda_2 \le \cdots \le \lambda_j \le \cdots$ be the eigenvalues of the above boundary value problem.
By the Min-Max principle, we have
\begin{eqnarray*}
\lambda_1 = \underset{\underset{u\neq 0}{u\in W^1_2(S_n)}}\min \frac{\mathcal{E}_{S_n}[u]}{\int_{S_n} \mid u(x)\mid^2\, dx},\\
\lambda_2 = \underset{\underset{u\neq 0, u \perp u_1}{u\in W^1_2(S_n)}}\min\frac{\mathcal{E}_{S_n}[u]}{\int_{S_n} \mid u(x)\mid^2\, dx},
\end{eqnarray*}
where $u_1$ is a normalized eigenfunction corresponding to $\lambda_1$. The function $u_1$ does not depend on $x_1$ and, viewed as
a function of one variable $x_2$, it is a normalized eigenfunction of $-\Delta = -\partial^2_{x_2}$ on $(0, a)$ with boundary conditions \eqref{R} corresponding to the first eigenvalue $\lambda_1$, moreover $\lambda_1 < \lambda_2$ (see \cite{Kar} or \cite[Section 1.5]{ME}).

It follows from the above that for all $u\in W^1_2(S_n),\;\; u \perp u_1$, one has
$$
 \lambda_2\int_{S_n}\mid u(x)\mid^2\,dx \le \mathcal{E}_{S_n}[u],
$$ which in turn implies
\begin{eqnarray*}
\mathcal{E}_{S_n}[u] - \lambda_1\int_{S_n}\mid u(x)\mid^2\,dx &=& \mathcal{E}_{S_n}[u] - \lambda_2\int_{S_n}\mid u(x)\mid^2\,dx\\ &+& (\lambda_2 - \lambda_1)\int_{S_n}|u(x)|^2\,dx\\&\geq&(\lambda_2 - \lambda_1)\int_{S_n}|u(x)|^2\,dx .
\end{eqnarray*}
Since $\lambda_1 < \lambda_2$, one gets
\begin{equation}\label{Rbeqn1}
\int_{S_n}|u(x)|^2\,dx \le \frac{1}{\lambda_2 - \lambda_1}\left(\mathcal{E}_{S_n}[u] - \lambda_1\int_{S_n}|u(x)|^2\,dx\right),\;\;\;\forall u\in W^1_2(S_n),\;\;u\perp u_1.
\end{equation}
\begin{lemma}\label{rbthm1}{\rm [Ehrling's Lemma]}
Let $X_0,\;X_1\; \textrm{and} \;X_2$ be Banach spaces such that $X_2 \hookrightarrow X_1$ is compact and $X_1\hookrightarrow X_0$. Then for every $\varepsilon > 0$, there exists a constant $C(\varepsilon) > 0$ such that
\begin{equation}\label{eqthm1}
\|u\|_{X_1} \le \varepsilon\|u\|_{X_2} + C(\varepsilon)\|u\|_{X_0}, \;\;\;\;\; \forall u\in X_2\,.
\end{equation}
\end{lemma}
See, e.g., \cite{MR} for details and proof.

Let $S(\mathbb{R}^2)$ be the class of all functions $\varphi\in C^{\infty}(\mathbb{R}^2)$ such that for any multi-index $\gamma$ and any $k\in\mathbb{N}$,
$$
\underset{x\in \mathbb{R}^2}\sup(1 + |x|)^k|\partial^{\gamma}\varphi(x)| < \infty .
$$Denote by $S'(\mathbb{R}^2)$ the dual space of $S(\mathbb{R}^2)$. For $s > 0$, let
$$
H^{s}(\mathbb{R}^2) := \left\{u\in S'(\mathbb{R}^2): \int_{\mathbb{R}^2}(1 + |\xi|^2)^{s}|\widehat{u}(\xi)|^2d\xi < \infty\right\},\;\;s\in\mathbb{R}.
$$Here, $\widehat{u}(\xi)$ is the Fourier image of $u(x)$ defined by
$$
\widehat{u}(\xi) = \frac{1}{2\pi}\int_{\mathbb{R}^2}e^{-ix\xi}u(x)dx.
$$

Let
$$
 H^s(S_n) := \left\{v = \tilde{v}|_{S_n}\;:\; \tilde{v}\in H^s(\mathbb{R}^2)\right\},
$$
$$
\|v\|_{H^s(S_n)} := \underset{\underset{\tilde{v}|_{S_n} = v}{\tilde{v}\in H^s\left({\mathbb{R}^2}\right)}}\inf \|\tilde{v}\|_{H^s(\mathbb{R}^2)}\;.
$$
Now, let $X_0 = L^2(S_n), X_1 = H^s(S_n)\; \textrm{for}\;\frac{1}{2} < s < 1 \; \textrm{and}\; X_2 = W^1_2 (S_n)$ in Lemma \ref{rbthm1}. That $X_2\hookrightarrow X_1$ is compact follows from the Sobolev compact embedding theorem (see, e.g., \cite[Ch. VII]{Ad} or  \cite[$\S$ 1.4.6]{Maz}). Thus we have the following lemma:
\begin{lemma}\label{rblemma0}
Let $\varepsilon > 0$ be given. Then there exists a constant $C_{1} >0$ such that
\begin{eqnarray}\label{rblemmaeqn}
&&\int_{n}^{n + 1}|u(x_1, a)|^2\,dx_1 + \int_n^{n + 1}|u(x_1, 0)|^2\,dx_1 \le C_{1}\left(\mathcal{E}_{S_n}[u] - \lambda_1\int_{S_n}|u(x)|^2\,dx\right),\nonumber\\
&&\forall u\in W^1_2(S_n),\; u\perp u_1.
\end{eqnarray}
\end{lemma}
\begin{proof}
In this proof, we make use of \eqref{Rbeqn1} and Lemma \ref{rbthm1}. For $s > \frac{1}{2}$, the trace theorem  and Lemma \ref{rbthm1} imply
\begin{eqnarray*}
&&\int_n^{n + 1}|u(x_1, a)|^2dx_1 + \int_n^{n + 1}|u(x_1, 0)|^2dx_1 \le C_s\|u\|_{X_1}\\ && \le C_s\left(\varepsilon\left(\int_{S_n}|\nabla u(x)|^2 dx + \int_{S_n}|u(x)|^2dx\right) + C(\varepsilon)\int_{S_n}|u(x)|^2 dx\right)\\ &&=C_s\varepsilon\int_{S_n}|\nabla u(x)|^2dx + C_s(\varepsilon + C(\varepsilon))\int_{S_n}|u(x)|^2dx\\&& = C_s\varepsilon\Big(\mathcal{E}_n[u] - \lambda_1\int_{S_n}|u(x)|^2 dx - \beta\int_n^{n + 1}|u(x_1, a)|^2\,dx_1\\&& \,+ \alpha\int_n^{n + 1}|u(x_1, 0)|^2\,dx_1 + \lambda_1\int_{S_n}|u(x)|^2\,dx\Big)  +  C_s(\varepsilon + C(\varepsilon))\int_{S_n}|u(x)|^2\,dx\\&&\le C_s\varepsilon\left(\mathcal{E}_{S_n}[u] - \lambda_1\int_{S_n}|u(x)|^2\,dx\right)\\&&\, + C_s\varepsilon \;\textrm{max}\{|\beta|,|\alpha|\}\left( \int_n^{n + 1}|u(x_1, a)|^2\,dx_1 + \int_n^{n + 1}|u(x_1, 0)|^2\,dx_1\right)\\&&  + \; C_s\left(\varepsilon(\lambda_1 + 1) + C(\varepsilon)\right)\int_{S_n}|u(x)|^2\,dx\,.
\end{eqnarray*}Take $\varepsilon\le \frac{1}{2C_s\max\{|\beta|,|\alpha|\}}$. Then
\begin{eqnarray*}
&&\int_n^{n + 1}|u(x_1, a)|^2\,dx_1 + \int_n^{n + 1}|u(x_1, 0)|^2\,dx_1 \\&&\le C_s\varepsilon\left(\mathcal{E}_{S_n}[u] - \lambda_1\int_{S_n}|u(x)|^2\,dx\right) + \frac{1}{2} \int_n^{n + 1}|u(x_1, a)|^2\,dx_1\\&& + \frac{1}{2}\int_n^{n + 1}|u(x_1, 0)|^2\,dx_1   + C_s\left(\varepsilon(\lambda_1 + 1) + C(\varepsilon)\right)\int_{S_n}|u(x)|^2\,dx.
\end{eqnarray*}
Hence \eqref{Rbeqn1} yields
\begin{eqnarray*}
&&\int_n^{n + 1}|u(x_1, a)|^2\,dx_1 + \int_n^{n + 1}|u(x_1, 0)|^2\,dx_1\\ &&\le 2C_s\varepsilon\left(\mathcal{E}_{S_n}[u] - \lambda_1\int_{S_n}|u(x)|^2\,dx\right) + 2C_s\left(\varepsilon(\lambda_1 + 1) + C(\varepsilon)\right)\int_{S_n}|u(x)|^2\,dx\\&&\le C_s\left(2\varepsilon + \frac{2}{\lambda_2 - \lambda_1}(\varepsilon(\lambda_1 + 1) + C(\varepsilon))\right)\left(\mathcal{E}_{S_n}[u] - \lambda_1\int_{S_n}|u(x)|^2\,dx\right)\\&& = C_{1}\left(\mathcal{E}_{S_n}[u] - \lambda_1\int_{S_n}|u(x)|^2\,dx\right),
\end{eqnarray*}where
$$
C_{1} := C_s\left(2\varepsilon + \frac{2}{\lambda_2 - \lambda_1}(\varepsilon(\lambda_1 + 1) + C(\varepsilon))\right).
$$
\end{proof}As a consequence of Lemma \ref{rblemma0} and \eqref{Rbeqn1} we have the following Lemma
\begin{lemma}\label{cor1}
There exists a constant $C_2 > 0$ such that
\begin{equation}\label{eqn0}
\int_{S_n}|\nabla u(x)|^2\,dx \le C_2\left(\mathcal{E}_{S_n}[u] - \lambda_1\int_{S_n}|u(x)|^2\,dx\right), \;\;\;\;\forall u\in W^1_2(S_n),\;u\perp u_1.
\end{equation}
\end{lemma}
\begin{proof}
\begin{eqnarray*}
\int_{S_n}|\nabla u(x)|^2\,dx &=& \mathcal{E}_{S_n}[u] - \lambda_1\int_{S_n}|u(x)|^2\,dx + \lambda_1\int_{S_n}|u(x)|^2\,dx \\&-& \beta \int_n^{n + 1}|u(x_1, a)|^2\,dx_1 + \alpha \int_n^{n + 1}|u(x_1, 0)|^2\,dx_1\\
&\le& \left(1 + C_{1}\max\{|\alpha|, |\beta|\} \right)\left(\mathcal{E}_{S_n}[u] - \lambda_1\int_{S_n}|u(x)|^2\,dx \right)\\ &+& \lambda_1\int_{S_n}|u(x)|^2\,dx \\
&\le& \left(1 + C_{1}\max\{|\alpha|, |\beta|\} \right)\left(\mathcal{E}_{S_n}[u] - \lambda_1\int_{S_n}|u(x)|^2\,dx \right)\\ &+& \frac{\max\{0,\lambda_1\}}{\lambda_2 - \lambda_1}\left(\mathcal{E}_{S_n}[u] - \lambda_1\int_{S_n}|u(x)|^2\,dx \right)\\
&=& \left(1 + C_{1}\max\{|\alpha|, |\beta|\} + \frac{\max\{0,\lambda_1\}}{\lambda_2 - \lambda_1} \right)\left(\mathcal{E}_{S_n}[u] - \lambda_1\int_{S_n}|u(x)|^2\,dx \right)\\&=& C_2\left(\mathcal{E}_{S_n}[u] - \lambda_1\int_{S_n}|u(x)|^2\,dx \right),
\end{eqnarray*}where
$$
C_2 := 1 + C_{1}\max\{|\alpha|, |\beta|\} + \frac{\max\{0,\lambda_1\}}{\lambda_2 - \lambda_1}\,.
$$
\end{proof}
Let
 $\mathcal{H}_1 := P W^1_2(S)$ and $\mathcal{H}_2 := (I- P)W^1_2(S)$, where
\begin{equation}\label{projection}
Pu(x) := \left(\int_0^au(x)\overline{u_1}(x_2)\,dx_2\right)u_1(x_2) = w(x_1)u_1(x_2)\,, \;\;\;\;\;\;\; \forall u\in W^1_2(S)
\end{equation}
and
$$
w(x_1):= \int_0^au(x)\overline{u_1}(x_2)\,dx_2\,.
$$Then $P$ is a projection since $P^2 = P$.
\begin{lemma}\label{rblemma1}
For all $u\in W^1_2(S), \;\;\langle (I - P)u(x_1,.), u_1\rangle_{L^2(0, a)} = 0$ for almost all $x_1 \in \mathbb{R}$.
\end{lemma}
\begin{proof}
Since $Pu = \langle u(x_1,.), u_1\rangle_{L^2(0, a) }u_1$, then
\begin{eqnarray*}
\langle (I - P)u(x_1,.), u_1\rangle_{L^2(0, a)} &=& \langle u(x_1,.) - \langle u(x_1,.), u_1\rangle_{L^2(0, a)} u_1, u_1\rangle_{L^2(0, a)}\\&=&\langle u(x_1,.), u_1\rangle_{L^2(0, a)} - \langle u(x_1,.), u_1\rangle_{L^2(0, a)} \langle u_1, u_1\rangle_{L^2(0, a)}\\&=& \langle u(x_1,.), u_1\rangle_{L^2(0, a)} - \langle u(x_1,.), u_1\rangle_{L^2(0, a)}  = 0.
\end{eqnarray*}
\end{proof}
\begin{lemma}\label{rblemma2}
For all $v\in\mathcal{H}_1, \; \tilde{v}\in\mathcal{H}_2,\;\;\langle v, \tilde{v}\rangle_{L^2(S)} = 0$ and  $\langle v_{x_1}, \tilde{v}_{x_1}\rangle_{L^2(S)} = 0$.
\end{lemma}
\begin{proof}
\begin{eqnarray*}
\langle \tilde{v}, v\rangle_{L^2(S)} &=& \int_S(I - P)u(x) . \overline{w(x_1)u_1({x_2)}}\,dx\\ &=&\int_{\mathbb{R}}\overline{w(x_1)}\left(\int_0^a u(x)\overline{u_1(x_2)}\,dx_2\right)dx_1\\&-&  \int_{\mathbb{R}}\overline{w(x_1)}\left[\left(\int_0^au\overline{u_1(x_2)}\,dx_2\right)\left(\int_0^a u_1(x_2)\overline{u_1(x_2)}\,dx_2\right)\right]dx_1\\&=&\int_{\mathbb{R}}\overline{w(x_1)}\left(\int_0^a u(x)\overline{u_1(x_2)}\,dx_2\right)dx_1\\ &-& \int_{\mathbb{R}}\overline{w(x_1)}\left[\left(\int_0^au(x)\overline{u_1(x_2)}\,dx_2\right)\|u_1\|^2\right]dx_1\\&=&\int_{\mathbb{R}}\overline{w(x_1)}\left(\int_0^a u(x)\overline{u_1(x_2)}\,dx_2\right)dx_1\\ &-& \int_{\mathbb{R}}\overline{w(x_1)}\left(\int_0^a u(x)\overline{u_1(x_2)}\,dx_2\right)dx_1 = 0.
\end{eqnarray*}
Since $(Pw)_{x_1} = P w_{x_1}$ for every $w \in W^1_2(S)$, one has
for all $v\in\mathcal{H}_1 \; \textrm{and}\; \tilde{v}\in\mathcal{H}_2$, $$v_{x_1}\in P L^2(S), \; \tilde{v}_{x_1}\in (I - P)L^2(S),$$ and hence
it follows from the above that
$$\langle v_{x_1}, \tilde{v}_{x_1}\rangle_{L^2(S)} = 0.$$
\end{proof}
\begin{lemma}\label{rblemma2*}
Let
$$
\mathcal{E}_S[u] := \int_S|\nabla u(x)|^2\,dx  + \beta \int_{\mathbb{R}}|u(x_1, a)|^2\,dx_1 - \alpha \int_{\mathbb{R}}|u(x_1, 0)|^2\,dx_1, \;\;\forall u \in W^1_2(S).
$$ Then
$$
\mathcal{E}_S[u] = \mathcal{E}_S[v] + \mathcal{E}_S[\tilde{v}],\;\;\;\forall u = v + \tilde{v},\,\,v\in\mathcal{H}_1\,,\;\;\tilde{v}\in\mathcal{H}_2\,.
$$
\end{lemma}
\begin{proof}
 \begin{eqnarray*}
 \langle\tilde{v}_{x_2}, v_{x_2}\rangle_{L^2(S)} &=& \int_{S}\frac{\partial}{\partial x_2}(I - P)u(x) \frac{\partial}{\partial x_2}(\overline{w(x_1)u_1(x_2)})dx\\&=& \int_{\mathbb{R}}\overline{w(x_1)}\left[\int_0^a\frac{\partial}{\partial x_2}(I - P)u(x) \frac{\partial}{\partial x_2}(\overline{u_1(x_2)})\,dx_2\right]dx_1.
 \end{eqnarray*}Integration by parts and Lemma \ref{rblemma1} give
 \begin{eqnarray*}
 \langle \tilde{v}_{x_2}, v_{x_2}\rangle_{L^2(S)} &=& \int_{\mathbb{R}}\overline{w(x_1)}(I -P)u(x_1, a)\frac{\partial}{\partial x_2}\overline{u_1(a)}\,dx_1\\ &-& \int_{\mathbb{R}}\overline{w(x_1)}(I - P)u(x_1, 0)\frac{\partial}{\partial x_2}\overline{u_1(0)}\,dx_1 \\&+& \int_{\mathbb{R}}\overline{w(x_1)}\left(\underbrace{\lambda_1\int_0^a(I - P)u(x)\overline{u_1(x_2)}\,dx_2}_{= 0}\right)dx_1\\&=&-\beta\int_{\mathbb{R}}\overline{w(x_1)}(I - P)u(x_1, a)\overline{u_1(a)}dx_1\\ &+& \alpha\int_{\mathbb{R}}\overline{w(x_1)}(I - P)u(x_1, 0)\overline{u_1(0)}dx_1.
 \end{eqnarray*}Thus, this together with Lemma \ref{rblemma2} yield
 \begin{eqnarray*}
 \mathcal{E}_S( \tilde{v}, v)&=&\int_S\nabla \tilde{v}\overline{\nabla v}\,dx + \beta\int_{\mathbb{R}}\overline{w(x_1)}(I - P)u(x_1, a)\overline{u_1(a)}dx_1\\ &-& \alpha\int_{\mathbb{R}}\overline{w(x_1)}(I - P)u(x_1, 0)\overline{u_1(0)}dx_1\\ &=& - \beta\int_{\mathbb{R}}\overline{w(x_1)}(I - P)u(x_1, a)\overline{u_1(a)}dx_1\\ &+& \alpha\int_{\mathbb{R}}\overline{w(x_1)}(I - P)u(x_1, 0)\overline{u_1(0)}dx_1\\&+&  \beta\int_{\mathbb{R}}\overline{w(x_1)}(I - P)u(x_1, a)\overline{u_1(a)}dx_1\\ &-& \alpha\int_{\mathbb{R}}\overline{w(x_1)}(I - P)u(x_1, 0)\overline{u_1(0)}dx_1 = 0.
 \end{eqnarray*}
 This means that for all $ u \in W^1_2(S)$
 $$
 \mathcal{E}_S[u] = \mathcal{E}_S[v] + \mathcal{E}_S[\tilde{v}], \;\;\;\;\forall u = v + \tilde{v},\,\, v\in\mathcal{H}_1,\;\;\tilde{v}\in\mathcal{H}_2.
 $$
 \end{proof}

\section{Proof of Threorem \ref{rbthm}}
Let
\begin{eqnarray*}
\mathcal{E}_S[u] &:=& \int_S|\nabla u(x)|^2\,dx - \alpha\int_{\mathbb{R}}|u(x_1, 0)|^2\,dx_1 + \beta \int_{\mathbb{R}}|u(x_1, a)|^2\,dx_1,\\
\textrm{Dom}(\mathcal{E}_S) &=& W^1_2(S).
\end{eqnarray*}
 and
\begin{eqnarray*}
\mathcal{E}_{\lambda_1,\mu, S}[u] &:=& \mathcal{E}_S[u] - \lambda_1 \int_{S}| u(x)|^2\,dx - \int_{\overline{S}} V(x)|u(x)|^2\,d\mu(x),\\
\textrm{Dom} (\mathcal{E}_{\lambda_1, \mu, S}) &=& W^1_2(S)\cap L^2\left(\overline{S}, Vd\mu\right).
\end{eqnarray*} Then  one has
\begin{equation}\label{radstrip}
N_-\left(\mathcal{E}_{\lambda_1,\mu, S}\right) \le N_-(\mathcal{E}_{1, 2\mu}) + N_-(\mathcal{E}_{2, 2\mu})\,
\end{equation}
where $\mathcal{E}_{1, 2\mu}$ and $\mathcal{E}_{2, 2\mu}$ are the restrictions of the form $\mathcal{E}_{\lambda_1,2\mu, S}$ to the spaces $\mathcal{H}_1$ and $\mathcal{H}_2$ respectively. We start by estimating the first term in the right-hand side of \eqref{radstrip}.
\\\\
Recall that for all $ u\in\mathcal{H}_1\,, \; u(x) = w(x_1)u_1(x_2)$ (see \eqref{projection}). Let $I$ be an arbitrary interval in $\mathbb{R}$ and let
$$
\nu(I) := \int_I\int_0^a V(x)|u_1(x_2)|^2\,d\mu(x).
$$
Then
\begin{eqnarray*}
\int_{\overline{S}} V(x)|u(x)|^2\,d\mu(x) &=& \int_{\mathbb{R}}\int_0^a V(x)|w(x_1)u_1(x_2)|^2\,d\mu(x)\\
&=& \int_{\mathbb{R}}|w(x_1)|^2\,d\nu(x_1)
=\int_{\mathbb{R}}|w(x_1)|^2\,d\nu(x_1).
\end{eqnarray*}

 On the subspace $\mathcal{H}_1$, one has
 \begin{eqnarray*}
 &&\int_S\left(\mid\nabla u(x)\mid^2 - \lambda_1\mid u(x)\mid^2\right)dx + \beta\int_{\mathbb{R}}\mid u(x_1, a)\mid^2dx_1\\&& - \alpha\int_{\mathbb{R}}\mid u(x_1, 0)\mid^2dx_1 - 2\int_{\overline{S}} V(x)\mid u(x)\mid^2d\mu(x)\\&&=
 \int_{\mathbb{R}}\mid w'(x_1)\mid^2\left(\int_0^a\mid u_1(x_2)\mid^2dx_2\right)dx_1\\&&
 + \int_{\mathbb{R}}\mid w(x_1)\mid^2\left(\int_0^a\mid u'_1(x_2)\mid^2dx_2\right)dx_1\\&&
 - \lambda_1\int_{\mathbb{R}}\mid w(x_1)\mid^2\left(\int_0^a\mid u_1(x_2)\mid^2dx_2\right)dx_1\\&&
 + \beta\int_{\mathbb{R}}\mid w(x_1)u_1(a)\mid^2dx_1 - \alpha\int_{\mathbb{R}}\mid w(x_1)u_1(0)\mid^2dx_1\\&&
 - 2\int_{\mathbb{R}}|w(x_1)|^2\,d\nu(x_1).
 \end{eqnarray*}But
 \begin{eqnarray*}
 &&\int_{\mathbb{R}}\mid w(x_1)\mid^2\left(\int_0^a\mid u'_1(x_2)\mid^2dx_2\right)dx_1\\&& = \lambda_1\int_{\mathbb{R}}\mid w(x_1)\mid^2\left(\int_0^a\mid u_1(x_2)\mid^2dx_2\right)dx_1 \\&& - \beta\int_{\mathbb{R}}\mid w(x_1)u_1(a)\mid^2dx_1 + \alpha\int_{\mathbb{R}}\mid w(x_1)u_1(0)\mid^2dx_1,
 \end{eqnarray*} which implies
 \begin{eqnarray}\label{stripH1}
 &&\int_S\left(\mid\nabla u(x)\mid^2 - \lambda_1\mid u(x)\mid^2\right)dx + \beta\int_{\mathbb{R}}\mid u(x_1, a)\mid^2dx_1 \nonumber\\&& - \alpha\int_{\mathbb{R}}\mid u(x_1, 0)\mid^2dx_1 - 2\int_{\overline{S}} V(x)\mid u(x)\mid^2d\mu(x) \nonumber\\&&=
 \|u_1\|^2\int_{\mathbb{R}}\mid w'(x_1)\mid^2dx_1 - 2\int_{\mathbb{R}}|w(x_1)|^2\,d\nu(x_1)\nonumber\\&&= \int_{\mathbb{R}}\mid w'(x_1)\mid^2dx_1 - 2\int_{\mathbb{R}}|w(x_1)|^2\,d\nu(x_1).
 \end{eqnarray}
 Hence, we have the following one-dimensional Schr\"odinger operator
 $$
 - \frac{d^2}{dx^2_1} - 2\nu\;\;\;\;\; \textrm{on} \;\;L^2(\mathbb{R})\,.
 $$
 Let
\begin{eqnarray*}
\mathcal{E}_{1,2\nu}[w] &:=& \int_{\mathbb{R}}|w'(x_1)|^2\,dx_1 - 2\int_{\mathbb{R}}|w(x_1)|^2\,d\nu(x_1),\\
\textrm{Dom} (\mathcal{E}_{1, 2\nu}) &=& W^1_2(\mathbb{R})\cap L^2\left(\mathbb{R}, d\nu\right),
\end{eqnarray*}

\begin{eqnarray*}
F_n &:=& \int_{I_n}|x_1|\,d\nu(x_1),\;\; n\neq 0,\\ F_0 &:=& \int_{I_0}d\nu(x_1).
\end{eqnarray*}
Then
\begin{equation}\label{radest}
N_-\left(\mathcal{E}_{1,2\nu}\right) \le 1 + 7.61 \underset{\{F_n > 0.046,\;n\in\mathbb{Z}\}}\sum \sqrt{F_n}
\end{equation}(see \cite[(2.42)]{ME}, see also the estimate before (39) in \cite{Eugene}).
To write the above estimate in terms of the original measure, let
\begin{eqnarray*}
\mathcal{F}_n &:=& \int_{I_n}\int_0^a|x_1|V(x)|u_1(x_2)|^2\,d\mu(x),\;\; n\neq 0,\\ \mathcal{F}_0 &:=& \int_{I_0}\int_0^a V(x)|u_1(x_2)|^2\,d\mu(x).
\end{eqnarray*}
Then $F_n = \mathcal{F}_n$. Hence
\begin{equation}\label{radest1}
N_-\left(\mathcal{E}_{1,2\mu}\right) \le 1 + 7.16 \underset{\{\mathcal{F}_n > 0.046,\;n\in\mathbb{Z}\}}\sum \sqrt{\mathcal{F}_n}\,.
\end{equation}

Next, we consider the subspace $\mathcal{H}_2 \subset W^1_2(S)$. By \eqref{Rbeqn1} and \eqref{eqn0}, one has
\begin{equation}\label{W1}
\|u\|^2_{W^1_2(S_n)} \le \left(\frac{1}{\lambda_2 - \lambda_1} + C_2\right)\left(\mathcal{E}_{S_n}[u] - \lambda_1\int_{S_n}|u(x)|^2dx\right)
\end{equation}for all $u\in W^1_2(S_n), \; u \perp u_1$.\\
Let $S_n := (n , n+ 1)\times (0, a)$, $n\in\mathbb{Z}$ with $\mu(\overline{S_n}) > 0$ be the set $G$ in Lemma \ref{measlemma4} and $S^*_n$ be defined as above
(see the paragraph after Corollary \ref{cor-direct}). For each $n$, $S^*_n$ intersects not more than $N_0$ rectangles to the left of $S_n$ and $N_0$ rectangles to right of $S_n$, where $N_0\in\mathbb{N}$ depends only on $a$ and $\theta_0$ in Corollary \ref{cor-direct}.
(It is not difficult to see that the side length of $S^*_n$ is less than or equal to $3\sqrt{a^2 +1}$ and hence
$\left[3\sqrt{2}\sqrt{a^2 +1}\right] + 1$ provides an upper estimate for $N_0$.) Then \eqref{locAhlfors*} implies
\begin{eqnarray*}
\mu(S^*_n) &\le& \sum_{j = n - N_0}^{n + N_0}\mu\left({\overline{S_j}}\right)\\ &=&
\mu\left(\overline{S_{n-N_0}}\right) + ... + \mu\left(\overline{S_{n-1}}\right) + \mu\left(\overline{S_n}\right) +
\mu\left(\overline{S_{n + 1}}\right) + ... + \mu\left(\overline{S_{n + N_0}}\right)\\
&\le& \left(\frac{1}{c^{N_0}_2} + ... + \frac{1}{c_2}\right)\mu\left(\overline{S_n}\right) +
\mu\left(\overline{S_n}\right) + \left(\frac{1}{c_2} + ... + \frac{1}{c^{N_0}_2}\right)\mu\left(\overline{S_n}\right)\\
&=&\left(2\left(\frac{1}{c_2} + ... + \frac{1}{c^{N_0}_2}\right)+ 1\right)\mu\left(\overline{S_n}\right)\\ &=&\kappa_0\mu\left(\overline{S_n}\right)\,,
\end{eqnarray*}
where $$\kappa_0 := 2\left(\frac{1}{c_2} + ... + \frac{1}{c^{N_0}_2}\right) + 1\,.$$

Now it follows from Lemma \ref{measlemma4} that for any $V \in L_{\mathcal{B}}(\overline{S_n}, \mu), \, V \ge 0$
$$
\int_{\overline{S_n}}V(x)|u(x)|^2d\mu(x) \le C_0 m^{-1}\|V\|_{\mathcal{B}, \overline{S_n}, \mu} \|u\|^2_{W^1_2(S_n)}
$$for all $u \in W^1_2(S_n) \cap C(\overline{S_n})$ satisfying the $m_0$ orthogonality conditions in Lemma \ref{measlemma4},
where the constant $C_0$ is independent of $V$, $m$, and $n$.
Hence \eqref{W1} implies
\begin{equation}\label{W2}
\int_{\overline{S_n}}V(x)|u(x)|^2d\mu(x) \le C_{3}m^{-1} \|V\|_{\mathcal{B}, \overline{S_n}, \mu}
\left(\mathcal{E}_{S_n}[u] - \lambda_1\int_{S_n}|u(x)|^2dx\right)
\end{equation} for all $u \in W^1_2(S_n) \cap C(\overline{S_n}), \, u \perp u_1$ satisfying the $m_0$ orthogonality conditions, where
$$
C_{3}:=  C_0 \left(\frac{1}{\lambda_2 - \lambda_1} + C_2\right).
$$
Let
\begin{eqnarray}\label{S_n}
&& \mathcal{E}_{2,2\mu, S_n}[u] :=
\mathcal{E}_{S_n}[u] - \lambda_1\int_{S_n}| u(x)|^2\,dx - 2\int_{\overline{S_n}}V(x)|u(x)|^2\,d\mu(x),\nonumber\\
&& \textrm{Dom} (\mathcal{E}_{2, 2\mu, S_n}) =  (I - P)W^1_2(S_n)\cap L^2\left(\overline{S_n}, V d\mu\right)
\end{eqnarray}
(see \eqref{n}).
Taking $ m = \left[2C_3 \|V\|_{\mathcal{B}, \overline{S_n}, \mu}\right] + 1 $ in \eqref{W2}, one has
\begin{equation}\label{radest2}
N_-\left(\mathcal{E}_{2,2\mu, S_n}\right) \le  C_{4}\|V\|_{\mathcal{B}, \overline{S_n}, \mu} + 2,\;\;\;\;\;\forall V \geq 0
\end{equation}
where $C_{4} := 2C_{3}$ (see \cite[Lemma 3.2.14]{ME}).
Again, taking $m = 1$ (and $m_0 = 0$; see Lemma \ref{measlemma4}) in \eqref{W2}, we get
$$
2\int_{\overline{S_n}}V(x)|u(x)|^2\,d\mu(x) \le C_{4}\|V\|_{\mathcal{B}, \overline{S_n}, \mu}\left(\mathcal{E}_{S_n}[u] -
\lambda_1\int_{S_n}|u(x)|^2dx\right)\,,
$$for all $u \in W^1_2(S_n) \cap C(\overline{S_n}), \, u \perp u_1$.  If $\|V\|_{\mathcal{B}, \overline{S_n}, \mu} \le \frac{1}{C_{4}}$, then
$$
N_-\left(\mathcal{E}_{2,2\mu, S_n}\right) = 0\,.
$$ Otherwise, \eqref{radest2} implies
$$
N_-\left(\mathcal{E}_{2,2\mu, S_n}\right)\le C_{5}\|V\|_{\mathcal{B}, \overline{S_n}, \mu}\,,
$$
where $C_{5} := 3 C_{4}$.\\\\ Let $M_n = \|V\|_{\mathcal{B}, \overline{S_n}, \mu}$ (see Section \ref{S3}).
Then for any $c \le \frac{1}{C_{4}}$, the variational principle (see \eqref{variation}) implies
\begin{equation}\label{radest3}
N_-\left(\mathcal{E}_{2,2\mu}\right)\le C_{5}\underset{\{M_n >\; c,\;n\in\mathbb{Z}\}}\sum M_n,\;\;\;\;\forall V \geq 0\,.
\end{equation}
 Thus \eqref{radstrip}, \eqref{radest1} and \eqref{radest3} imply \eqref{rbtheqn}.

\section{Concluding remarks}
\begin{remark}
{\rm Recall that a sequence $\{a_n\}$ belongs to the ``weak $l_1$-space'' (Lorentz space) $l_{1,w}$ if the following quasinorm
\begin{equation}\label{quasi}
\|\{a_n\}\|_{1,w} = \underset{s > 0}\sup\left(s \;\textrm{card}\{n\;:\;|a_n| > s\}\right)
\end{equation} is finite. It is a quasinorm in the sense that it satisfies the weak version of the triangle inequality:
$$
\|\{a_n\} + \{b_n\}\|_{1,w} \le 2\left(\|\{a_n\}\|_{1,w} + \|\{b_n\}\|_{1,w}\right)
$$  (see, e.g., \cite{lf} for more details).}
\end{remark}

\begin{theorem}\label{sthm1}{\rm(cf. \cite[Theorem 9.2]{Eugene})}
Let $V \ge 0$. If $N_-\left(\mathcal{E}_{\lambda_1,\gamma \mu, S}\right) = O(\gamma)\mbox{ as } \gamma \longrightarrow +\infty$, then $\|\mathcal{F}_n\|_{1, w} < \infty$.
\end{theorem}
\begin{proof}
Consider the function
$$
w_n(x_1) := \left\{\begin{array}{cl}
  0 ,   & \  x_1 \le 2^{n-2}\,\textrm{ or }\, x_1\ge 2^{n+1} , \\ \\
  4(x_1 - 2 ^{n-2}) ,  & \ 2^{n-2} <  x_1 < 2^{n - 1}  , \\ \\
 2^n , & \ 2^{n - 1} \le x_1 \le  2^{n}, \\ \\
 2^{n+1} - x_1  ,  & \ 2^{n} <  x_1 < 2^{n+1}\,,
\end{array}\right.
$$
$n > 0$. Let $v_n(x) = w_n(x_1)u_1(x_2)$.
Then by a computation similar to the one leading to \eqref{stripH1} we get
\begin{eqnarray*}
\mathcal{E}_S[v_n] &-& \lambda_1\int_{S}|v_n(x)|^2\,dx =\int_S\left(|\nabla v_n(x)|^2 - \lambda_1|v_n(x)|^2\right)dx\\ &-& \alpha \int_{\mathbb{R}}|v_n(x_1, 0)|^2dx_1 + \beta \int_{\mathbb{R}}|v_n(x_1, a)|^2dx_1 = \int_{\mathbb{R}}|w'_n(x_1)|^2dx_1\\&=&
\int_{2^{n-2}}^{2^{n - 1}}|w'_n(x_1)|^2dx_1 + \int_{2^{n - 1}}^{2^{n}}|w'_n(x_1)|^2dx_1+\int_{2^{n}}^{2^{n+1}}|w'_n(x_1)|^2dx_1\\ &=&
4\cdot 2^n + 0 + 2^n = 5\cdot 2^n.
\end{eqnarray*}
Now
\begin{eqnarray*}
\int_{\overline{S}} V(x)|v_n(x)|^2\,d\mu(x) &\ge& \int_{2^{n - 1}}^{2^{n}}\int_0^a V(x) 2^{2n} |u_1(x_2)|^2\,d\mu(x)\\ &\ge&
2^n\int_{2^{n - 1}}^{2^{n}}\int_0^a |x_1| V(x) |u_1(x_2)|^2\,d\mu(x)\\&=& 2^n \mathcal{F}_n\,.
\end{eqnarray*}
It follows from the above that $\mathcal{E}_{\lambda_1, \mu, S}[v_n] < 0$ if $\mathcal{F}_n > 5$, $n > 0$.

One can  define functions $v_n$ for $n \le 0$
similarly to the above and extend to them the previous estimate.
The fact that $v_n$ and $v_k$ have disjoint supports if $|m - k| \ge 3$ implies that
$$
N_-\left(\mathcal{E}_{\lambda_1, V\mu, S}\right) \ge \frac{1}{3}\textrm{card}\{n\in\mathbb{Z}\,:\, \mathcal{F}_n > 5\}
$$ (see \cite[Theorem 9.1]{Eugene}). If $N_-\left(\mathcal{E}_{\lambda_1, \gamma \mu, S}\right) \le C\gamma$, then
$$
\frac{1}{3}\textrm{card}\{n\in\mathbb{Z}\,:\, \gamma\mathcal{F}_n > 5\} \le C\gamma\,,
$$which implies
$$
\textrm{card}\left\{n\in\mathbb{Z}\,:\, \mathcal{F}_n > \frac{5}{\gamma}\right\} \le 3 C\gamma\,.
$$ With $s = \frac{5}{\gamma}$, we have
$$
\textrm{card}\{n\in\mathbb{Z}\,:\, \mathcal{F}_n > s\} \le C_{7}s^{-1},\,\,\,s > 0,
$$ where $C_{7}:= 15\, C$.
\end{proof}

\begin{remark}
{\rm Suppose that $\mu = |\cdot|$, the  Lebesgue measure. Then
\begin{eqnarray*}
 \mathcal{F}_n &=&   \int_{I_n}|x_1|\left(\int_0^aV(x)|u_1(x_2)|^2dx_2\right)dx_1,  \;\;\;\; n\neq 0, \\
 \mathcal{F}_0 &=&  \int_{I_0}\left(\int_0^aV(x)|u_1(x_2)|^2dx_2\right)dx_1.
 \end{eqnarray*} Let $J_n := (n, n + 1),\;\;I := (0, a)$ and write $\parallel V\parallel_{\mathcal{B}, S_n}$ and $N_-\left(\mathcal{E}_{\lambda_1, S}\right)$ instead of $\parallel V\parallel_{\mathcal{B}, S_n, |\cdot|}$ and $N_-\left(\mathcal{E}_{\lambda_1,|\cdot|, S}\right)$ respectively. Further, let
 $$
 \mathcal{D}_n := \int_{J_n}\parallel V\parallel_{\mathcal{B}, I}\, dx_1\,.
 $$  Then, using \cite[Lemma 7.6]{Eugene} in place of our Lemma \ref{measlemma4},
 one gets
 \begin{equation}\label{radest4}
N_-\left(\mathcal{E}_{\lambda_1, S}\right)\le 1 + 7.61 \underset{\{\mathcal{F}_n > \;0.046,\;n\in\mathbb{Z}\}}\sum \sqrt{\mathcal{F}_n} +  C_{8}\underset{\{\mathcal{D}_n >\; c,\;n\in\mathbb{Z}\}}\sum \mathcal{D}_n,\;\;\;\;\forall V \geq 0\,.
\end{equation}
This estimate is stronger than \eqref{rbtheqn}. Indeed, suppose that $\|V\|_{(\mathcal{B}, S_n)}  = 1$. Since $\mathcal{B}(V)$ satisfies the $\Delta_2$-condition, then $\int_{S_n}\mathcal{B}(V(x))\,dx = 1$ (see (9.21) in \cite{KR}). Using \eqref{Luxemburgequiv} and \eqref{LuxNormPre}, we have
\begin{eqnarray}\label{imp}
\mathcal{D}_n = \int_{J_n}\|V\|_{\mathcal{B}, I}\,dx_1 &\le& 2\int_{J_n}\|V\|_{(\mathcal{B}, I)}\,dx_1 \nonumber\\ &\le& 2\int_n^{n + 1}\left(1 +\int_0^a\mathcal{B}\left(V(x)\right)dx_2\right)dx_1\nonumber\\
&=& 2 + 2\int_{J_n}\int_0^a\mathcal{B}\left(V(x)\right)dx = 4 \nonumber\\ &=& 4\|V\|_{(\mathcal{B}, S_n)} \le 4\|V\|_{\mathcal{B}, S_n}\nonumber\\ &=& 4M_n .
\end{eqnarray}
The scaling $V \longmapsto t V,\; t > 0$, allows one to extend the above inequality to an arbitrary $V \geq 0$.

 By the same procedure as the one leading to (56) in \cite{KM},  one has the following estimate
\begin{equation}\label{est14}
N_-(\mathcal{E}_{\lambda_1, S})\leq 1 + C_{9}\left( \parallel(\mathcal{F}_n)_{n\in\mathbb{Z}}\parallel_{1,w} + \|V_{*}\|_{L_1\left(\mathbb{R}, L_{\mathcal{B}}(I)\right)}\right), \;\;\;\ \forall V \geq 0,
\end{equation}
where
\begin{eqnarray*}
V_{*} &:=& V(x) - G(x_1),\\
G(x_1) &:=& \int_0^aV(x)|u_1(x_2)|^2\,dx_2\,,\\
\|V_{*}\|_{L_1\left(\mathbb{R}, L_{\mathcal{B}}(I)\right)} &:=& \int_{\mathbb{R}}\|V_*\|_{\mathcal{B}, I}\,dx_1\,.
\end{eqnarray*}
Estimates \eqref{radest4} and \eqref{est14} are equivalent to each other but the advantage of the latter is that it separates the contribution to the eigenvalues of $V(x) = V(x_1) $ from that of $V(x) = V(x_2)$.
The condition $\|(\mathcal{F}_n)\|_{{1, w}} < \infty$ is necessary and sufficient for the semi-classical behaviour of the estimate coming from the subspace $\mathcal{H}_1$ (see Theorem \ref{sthm1} above).  In addition, if $V_* \in L_1\left(\mathbb{R}, L_{\mathcal{B}}(I)\right)$, then one gets an analogue of \cite[Theorem 1.1]{LapSolo}, i.e.,
$$
N_-(\mathcal{E}_{\lambda_1, \gamma \mu, S}) = O(\gamma) \;\;\textrm{as}\;\;\gamma \longrightarrow\; +\infty
$$ if and only if $\mathcal{F}_n \in l_{1, w}$.
}
\end{remark}
\begin{remark}\label{dir}
{\rm One can think of the Dirichlet boundary conditions as the limit of the boundary conditions in \eqref{R} as
$\alpha$ and $\beta$ tend to infinity. In this case,
$$\lambda_1 = \frac{\pi^2}{a^2},\,\,\, u_1(x_2) = \sin\frac{\pi}{a}x_2 , \mbox{ and }
\lambda_2 = \min\left\{4\frac{\pi^2}{a^2}, \frac{\pi^2}{a^2} + \pi^2\right\} > \lambda_1.$$
Let $X_n := \{u\in W^1_2(S_n)\,:\, u(x_1, 0) = u(x_1, a) = 0\}$. Then for all $u\in X_n, u\perp \sin\frac{\pi}{a}x_2$, one has an analogue of  \eqref{Rbeqn1}
\begin{equation}\label{Dir1}
\int_{S_n}|u(x)|^2\,dx \le \frac{1}{\pi^2} \max\left\{\frac{a^2}{3}\, , 1\right\}\left(\int_{S_n}|\nabla u(x)|^2\,dx - \frac{\pi^2}{a^2}\int_{S_n}|u(x)|^2\,dx\right).
\end{equation}
Also, similarly to Lemma \ref{cor1}, there is a constant $C > 0$ such that
\begin{equation}\label{Dir2}
\int_{S_n}|\nabla u(x)|^2\,dx \le C\left(\int_{S_n}|\nabla u(x)|^2dx - \frac{\pi^2}{a^2}\int_{S_n}|u(x)|^2\,dx\right), \;\forall u\in X_n,\;u\perp \sin\frac{\pi}{a}x_2 .
\end{equation}
Now, for all $u \in X := \{u\in W^1_2(S)\,:\, u(x_1, 0) = u(x_1, a) = 0\}$, let
$$
Pu(x_1, t) := \left(\frac{2}{a}\int^a_0u(x_1, t)\sin\frac{\pi}{a}t\,dt\right)\sin\frac{\pi}{a}x_2\,.
$$
Then $P : X\longrightarrow X$ an orthogonal projection (cf. Lemma \ref{rblemma2*}).
Let $X_1 : = PX$ and $X_2 := (I - P)X$. Furthermore, let
\begin{eqnarray*}
q_{\mu, S}[u] &:=& \int_S\left|\nabla u(x)\right|^2dx - \frac{\pi^2}{a^2}\int_S|u(x)|^2dx
 - \int_{\overline{S}} V(x)|u(x)|^2\,d\mu(x),\\
\textrm{Dom}(q_{\mu, S}) &=& X\cap L^2\left(\overline{S}, Vd\mu\right).
\end{eqnarray*}
Then similarly to \eqref{radstrip}, we have
\begin{equation}\label{Dir3}
N_-\left(q_{\mu, S}\right) \le N_-(q_{1, 2\mu}) + N_-(q_{2, 2\mu})\,
\end{equation}
where $q_{1, 2\mu}$ and $q_{2, 2\mu}$ are the restrictions of the form $q_{2\mu, S}$ to the subspaces $X_1$ and $X_2$ respectively.
For an arbitrary interval $I$ on $\mathbb{R}$, let
$$
\nu(I) := \frac{2}{a}\int_I\int^a_0V(x)\sin^2\frac{\pi}{a}x_2\,d\mu(x).
$$
Then on the subspace $X_1$, a procedure similar to the one leading to \eqref{radest1} gives an estimate for the first term in \eqref{Dir3}, where in this case $\mathcal{F}_n$ is given by
\begin{eqnarray*}
\mathcal{F}_n &=& \frac{2}{a}\int_{I_n}\int_0^a|x_1|V(x)\sin^2\frac{\pi}{a}x_2\,d\mu(x),\,\,\, n\not= 0,\\
\mathcal{F}_0 &=& \frac{2}{a}\int_{I_0}\int_0^a|V(x)\sin^2\frac{\pi}{a}x_2\,d\mu(x).
\end{eqnarray*}
On the subspace $X_2$, it follows from \eqref{Dir1} and \eqref{Dir2} that there is a constant $C' > 0$ such that
$$
\|u\|^2_{X_n} \le C' \left(\int_{S_n}\left|\nabla u(x)\right|^2dx - \frac{\pi^2}{a^2}\int_{S_n} |u(x)|^2dx\right),\,\,\,\,
\forall u\in X_n,\,u \perp\sin\frac{\pi}{a}x_2 .
$$
Thus, we obtain, similarly  to \eqref{radest3}, an estimate for the second term in \eqref{Dir3}.
}
\end{remark}
\section{Acknowledgments}
The author acknowledges the funding from the Commonwealth Scholarship Commission in the UK, grant UGCA-2013-138. The author is also very grateful to Eugene Shargorodsky for his valuable guidance and discussions during the preparation of this work.


\begin{thebibliography}{1}
\bibitem{Ad}R. A. Adams,
{\em Sobolev Spaces.} Academic Press, New York 1975.
\bibitem{Aize} M. Aizenman and B. Simon, Brownian motion and Harnack's inequality for Schr\"odinger operators, {\it Comm. Pure Appl. Math.}, \textbf{35} (1982), 209--273.
\bibitem{BEL} A.A. Balinsky, W.D. Evans and  R.T. Lewis, {\em The Analysis and Geometry of Hardy's Inequality.} Universitext, Springer, Cham, 2015.
\bibitem{BerShu} F.A. Berezin and M.A. Shubin,
{\em The Schr\"odinger Equation.} Kluwer, Dordrecht etc., 1991.
\bibitem{BirSol} M.Sh. Birman and M.Z. Solomyak,
{\em Spectral Theory of Self-Adjoint Operators in Hilbert Space.} Kluwer, Dordrecht etc., 1987.
\bibitem{Bra} J. F. Brasche, P. Exner, Yu.A. Kuperin and P. \v{S}eba, Schr\"odinger operators with singular interactions, {\it J. Math. Anal. Appl.}, \textbf{184}, 1 (1994), 112--139.
\bibitem{Cou} R. Courant and D. Hilbert, {\em Methods in Mathematical Physics.} Vol. 1, Interscience Publishers. Inc., New York, 1966.
\bibitem{TE} T. Ekholm and H. Kova\v{r}\'{i}k, Stability of the magnetic Schr\"odinger operator in a waveguide, {\it Comm. PDE}, \textbf{30}, 4 (2005), 539--565.
\bibitem{PA} P. Exner and A. Minakov, Curvature-induced bound states in Robin waveguides and their asymptotical properties, {\it J. Math. Phys.}, (2014), 122101.
\bibitem{lf}L. Grafakos, {\em Classical Fourier Analysis.} Springer, New York, 2008.
\bibitem{Grig} A. Grigor'yan and N. Nadirashvili, Negative eigenvalues of two-dimensional
Schr\"odinger operators, {\it Arch. Rational Mech. Analy.}, \textbf{217} (2015), 975--1028.


\bibitem{Her} J. Herczynski, On Schr\"odinger operators with a distributional potential, {\it J. Operator Theory}, \textbf{21} (1989), 273--295.
\bibitem{ME} M. Karuhanga, Estimates for the number of  eigenvalues of two-dimensional Schr\"odinger operators lying below the essential spectrum, PhD-thesis, King's College London, (2016) (see also arxiv: 1609.08098).
\bibitem{Kar} M. Karuhanga, On the spectrum of the Laplacian on a strip with various boundary conditions, {\it Far East J. Math. Sci. (FJMS)}, \textbf{102}, 8 (2017), 1663--1675.
    
\bibitem{KM}  M. Karuhanga, Eigenvalue bounds for a class of Schr\"odinger operators in a strip, {\it J. Math.}, (2018), 8 pages.



\bibitem{KS} M. Karuhanga and E. Shargorodsky,
On negative eigenvalues of two-dimensional Schr\"odinger operators with singular potentials, (to appear).

\bibitem{KR} M.A. Krasnosel'skii and Ya.B. Rutickii,
{\em Convex Functions and Orlicz Spaces.} P. Noordhoff,
Groningen, 1961.
\bibitem{KK} D. Krej\v{c}i\v{r}\'{\i}k and J. K\v{r}\'{\i}\v{z}, On the spectrum of curved planar waveguides, {\it Publ. RIMS, Kyoto Univ.}, \textbf{41} (2005), 757--791.


\bibitem{LapSolo} A. Laptev and M. Solomyak,  On spectral estimates for two-dimensional Schr\"odinger operators,
{\it J. Spectr. Theory}, \textbf{3}, 4 (2013), 505--515.
\bibitem{Maz} V.G. Maz'ya,
{\em Sobolev Spaces. With Applications to Elliptic Partial Differential Equations.}
Springer, Berlin--Heidelberg, 2011.
\bibitem{Rad} R. Nov\'{a}k, Bound states in wavaguides with complex Robin boundary conditions, {\it Asymptotic analysis}, \textbf{96}, 3--4 (2016), 251--281.
\bibitem{RR} M.M. Rao and Z.D. Ren,
{\em Theory of Orlicz Spaces.}
Marcel Dekker, New York, 1991.

\bibitem{MR} M. Renardy and R. C. Rogers {\em An Introduction to Partial Differential Equations}. Springer-Verlag, Berlin, 1992.
\bibitem{AFR} A. F. Rossini, On the spectrum of a Robin Laplacian in a planar waveguide, {\it CMJ}, (2018), 1--17.


\bibitem{Roz}G. V. Rozenblum, The distribution of the discrete spectrum for singular differential operators, {\it Dokl. Akad. Nauk SSSR}, \textbf{202}(1972), 1012--1015.


\bibitem{Eugene} E. Shargorodsky,
On negative eigenvalues of two-dimensional Schr\"odingers operators, {\it Proc. London. Math. Soc.}, \textbf{108}, 2 (2014), 441--483.

\bibitem{Sol} M. Solomyak,
Piecewise-polynomial approximation of functions from $H\sp \ell((0,1)\sp d)$, $2\ell=d$, and applications to the spectral theory of the Schr\"odinger operator,
{\it Isr. J. Math.},  \textbf{86}, 1-3 (1994), 253--275.

\bibitem{Stein} E. M. Stein, {\em Singular Integrals and Differentiability Properties of Functions.} Princeton University Press, New Jersey, 1970.
\bibitem{Herc} P. Stollmann and J. Voigt, Pertubation of Dirichlet forms by measures, {\it Potential Analysis}, \textbf{5 }(1996), 109--138.

\end{thebibliography}
\end{document}